\documentclass[a4paper]{article}
\usepackage{amsthm,amssymb,amsmath,enumerate}
\usepackage{algorithm}

\usepackage{tikz}
\usetikzlibrary{backgrounds}
\usetikzlibrary{decorations}
\usetikzlibrary{decorations.pathmorphing}

\newtheorem{definition}{Definition}
\newtheorem{proposition}[definition]{Proposition}
\newtheorem{theorem}[definition]{Theorem}

\newtheorem{lemma}[definition]{Lemma}

\newcommand{\card}[1]{|#1|}

\newcommand{\comment}[1]{}

\newcommand{\es}{\emptyset}

\newcommand{\set}[1]{\left\{#1\right\}}

\newcommand{\emtext}[1]{\text{\em #1}}

\newcommand{\sm}{\setminus}

\makeatletter
\newcommand{\eqnum}{\leavevmode\hfill\refstepcounter{equation}\textup{\tagform@{\theequation}}}
\makeatother

\tikzstyle{hvertex}=[thick,circle,inner sep=0.cm, minimum size=2mm, fill=white, draw=black]
\tikzstyle{hedge}=[very thick]
\tikzstyle{rededge}=[ultra thick,red]
\tikzstyle{pathedge}=[hedge,decorate, decoration={random steps,segment length=3pt,amplitude=1pt}]

\title{$h$-perfect plane triangulations}
\author{Yohann Benchetrit and Henning Bruhn}
\date{}

\begin{document}
\maketitle

\begin{abstract}
We characterise $t$-perfect plane triangulations by forbidden induced subgraphs. As a consequence, we obtain that a plane triangulation is $h$-perfect if and 
only if it is perfect. 
\end{abstract}

\section{Introduction}

As defined by Berge, a graph is \emph{perfect} if for each of its induced subgraphs, the 
chromatic number is equal to the clique number. Results of  
Fulkerson~\cite{Fulkerson72}, Lov\'asz \cite{Lovasz1972} and Chv\'atal~\cite{Chvatal75} showed  that perfection could equally be characterised in polyhedral terms:
a graph is perfect if and only its
stable set polytope (the convex hull of its stable sets) is fully 
described by non-negativity and clique constraints (we defer 
precise definitions 
to the 
next section). 

The polyhedral setting for perfection suggested several generalisations.
A graph is \emph{$h$-perfect} if its stable set polytope is defined by non-negativity,
clique and odd-cycle constraints;  the $K_4$-free $h$-perfect graphs are called \emph{$t$-perfect}.

In this note, we show that $h$-perfection and perfection are equivalent for plane triangulations.
\begin{theorem}\label{hptrigthm}
A plane triangulation is $h$-perfect if and only if it is perfect. 
\end{theorem}
 To our knowledge, this is the first example of a rather wide class for which these two notions coincide. 
The equivalence is a consequence of a characterisation of $t$-perfect plane triangulations,
see Theorem~\ref{trigthm} below. 
This in itself is remarkable, 
as a characterisation of $h$-perfection in a class of graphs  does not normally 
follow directly from one for $t$-perfection.

The Strong Perfect Graph theorem~\cite{SPGT} gives a structural 
characterisation of perfect graphs in terms of minimal imperfect obstructions:
a graph is perfect if and only if it contains neither an odd hole
(an induced odd cycle other than a triangle) nor the complement of an odd hole.
An analogous characterisation for $t$-perfect or $h$-perfect graphs is not known.
Our second main result is such a characterisation of $t$-perfect
plane triangulations.

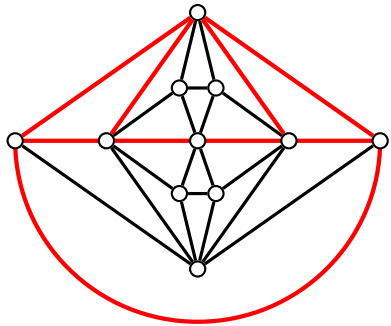
\begin{figure}[ht]
\centering
\begin{tikzpicture}[scale=1]

\def\xstep{1.2}
\def\ystep{1}

\draw[rededge] (0,0) -- (4*\xstep,0);
\foreach \i in {0,...,4}{
  \node[hvertex] (v\i) at (\i*\xstep,0){};
}

\node[hvertex] (A) at (1.8*\xstep,0.7*\ystep){};
\node[hvertex] (B) at (2.2*\xstep,0.7*\ystep){};
\node[hvertex] (C) at (2*\xstep,1.7*\ystep){};
\draw[hedge] (A) -- (B) -- (C) -- (A);
\draw[hedge] (A) edge (v1) edge (v2);
\draw[hedge] (B) edge (v2) edge (v3);
\draw[rededge] (C) edge (v0) edge (v1) edge (v3) edge (v4);

\node[hvertex] (a) at (1.8*\xstep,-0.7*\ystep){};
\node[hvertex] (b) at (2.2*\xstep,-0.7*\ystep){};
\node[hvertex] (c) at (2*\xstep,-1.7*\ystep){};
\draw[hedge] (a) -- (b) -- (c) -- (a);
\draw[hedge] (a) edge (v1) edge (v2);
\draw[hedge] (b) edge (v2) edge (v3);
\draw[hedge] (c) edge (v0) edge (v1) edge (v3) edge (v4);

\begin{scope}[on background layer]
\draw[rededge,bend right=90] (v0) arc (180:360:2*\xstep);
\end{scope}
\end{tikzpicture}

\caption{A plane triangulation that is $t$-imperfect but that does not contain
any induced odd wheel. A loose odd wheel is shown in red.}
\label{nooddwheelfig}
\end{figure}

Which (minimal) $t$-imperfect graphs could occur in a plane triangulation?
Odd wheels 
form a family of minimally t-imperfect graphs (see
for instance Schrijver~\cite[Ch.~68]{LexBible}), and they are planar. 
Consequently, a planar $t$-perfect graph cannot contain any odd wheel as 
an induced subgraph ($t$-perfection is closed under taking induced subgraphs). 
Excluding odd wheels, however, is not enough to ensure $t$-perfection, 
even for plane triangulations (see Figure~\ref{nooddwheelfig}). 
Rather, we have to  exclude certain subdivisions of odd wheels as well. 

For this,
we call a \emph{loose odd wheel} any 
graph that is obtained from an odd wheel by a subdivision of the edges 
of the rim, where the number of edges that are subdivided
an even number of times is odd, and at least three. (Again, more precise
definitions follow in the next section.)  
We prove:

\begin{theorem}\label{trigthm}
For every plane triangulation $T$, the following statements are equivalent:
\begin{enumerate}[\rm (i)]
\item\label{M1} $T$ is $t$-perfect;
\item\label{M2}  $T$ does not contain any loose odd wheel as an induced subgraph; and 
\item\label{M3} $T$ is perfect and $K_4$-free.
\end{enumerate}
\end{theorem}

Combining the two theorems, it is straightforward to show  that {a plane triangulation is $h$-perfect if and only if it does not contain any loose odd wheel other than $K_4$ as an induced subgraph}.

\smallskip
The  $h$-perfect graphs  have received 
some attention due to their algorithmic 
properties, which they share to some degree with perfect graphs. In particular, 
Gr\"otschel, Lov\'asz and Schrijver \cite{Grotschel1986} showed that 
a maximum-weight stable set of an $h$-perfect graph can be found in polynomial-time using the Ellipsoid method. 
Later Eisenbrand et al.~\cite{EFGK02} described  a combinatorial algorithm for $t$-perfect graphs. 
 The class of $t$-perfect graphs is of independent interest as it is the class of graphs whose fractional stable set polytope (that is the polytope defined by non-negativity and edge inequalities) has Chv\'atal rank at most~$1$ (see \cite{LexBible}).

Key questions about $t$-perfect or $h$-perfect graphs are: 
How can they be described
in terms of forbidden substructures? Can they always be coloured with few
colours? Is it possible to recognise them in polynomial time? These questions
have been pursued in a number of works, among 
them~\cite{Gerards89,Shepherd95,GS98,BS12,YohPhD}. 
See also Schrijver~\cite[Ch.~68]{LexBible} for more references.

In the special case of plane triangulations, Theorem~\ref{trigthm} answers the first question and  Theorem~\ref{hptrigthm} immediately answers the other ones: perfect graphs can always be coloured with no more colours than the size of a maximum clique, and they can be recognised in polynomial-time (the planar case is considerably simpler; see Tucker~\cite{Tucker73}).

\section{Basic definitions and facts}\label{sec:defs}

We keep to notation as in Diestel~\cite{diestelBook10}, where also 
all general definitions that we omit here may be found.  We only consider 
graphs that are
undirected, simple and finite.

A graph is \emph{perfect} if for each induced subgraph, the chromatic number and clique number are equal. For each graph $G$, consider the following inequalities over $\mathbb{R}^{V(G)}$:
\begin{align}
\label{nncnstr} x_v\geq 0 &\,\,\text{for every }v\in V(G)\\
\label{edgecnstr} x_u+x_v\leq 1 & \text{ for every edge }uv\in E(G)\\ 
\label{cliquecnstr}
\sum_{v\in K}x_v\leq 1 
& \text{ for every clique }K
\text{ of }G\\
\label{oddcnstr} \sum_{v\in V(C)}x_v\leq \frac{ |V(C)|-1}{2} 
& \text{ for every induced odd cycle }C
\text{ in }G
\end{align}

\noindent
{The graph $G$ is perfect if and only if the polytope described by \eqref{nncnstr} and \eqref{cliquecnstr} is integral} (that is, if all the vertices have integer coordinates only). 

A graph is \emph{$h$-perfect} if the polytope described by the inequalities \eqref{nncnstr}, \eqref{cliquecnstr} and \eqref{oddcnstr} is integral. This is precisely the
case when the polytope is the convex hull of the incidence vectors of the stable sets of $G$. 
A graph is \emph{$t$-perfect} if \eqref{nncnstr}, \eqref{edgecnstr} and \eqref{oddcnstr} define an integral polytope.

{Perfect graphs are $h$-perfect} (since they cannot contain induced odd cycles of length greater than~$3$), and it is straightforward to show that {$t$-perfect graphs are exactly
the $h$-perfect $K_4$-free graphs}. 

As for perfect graphs, {each induced subgraph of an $h$-perfect graph is $h$-perfect} (this follows from an easy polyhedral argument), and hence the same holds for $t$-perfection.

A \emph{hole} is an induced cycle of length at least~$4$; an
\emph{antihole} is the complement of a hole. 
Odd holes and odd antiholes are imperfect, and indeed
the Strong Perfect Graph theorem~\cite{SPGT} asserts
that these are the only \emph{minimally} imperfect graphs (in this note, \emph{minimally} is with respect to deletion of vertices): 
{a graph is  perfect if and only if it does not contain
an odd hole nor an odd antihole} (as an induced subgraph). 
A similar structural characterisation of $t$-perfection or $h$-perfection
is not known. 

We will use twice a well-known lemma in the context of 
perfect graphs. Its proof is straightforward from Berge's original definition of perfection.
\begin{lemma}[Chv\'atal~\cite{Chvatal75}]\label{cliquecutset}
Let $G$ be a graph and $G_1,G_2$ be two proper induced 
subgraphs of $G$  such that $G=G_1\cup G_2$ and
$G_1\cap G_2$ is a clique. Then $G$ is perfect if and only if $G_1$ and $G_2$
are perfect. 
\end{lemma}

A \emph{wheel} is a graph formed by a cycle $C$ (the \emph{rim}) together with a vertex $v$ (the \emph{center}) adjacent to every vertex of $C$. A wheel is \emph{odd} if it has an odd rim. 

It is easy to show that {odd wheels are minimally $t$-imperfect}; see
for instance Schrijver~\cite[Ch.~68]{LexBible}. 
Other minimally  $t$-imperfect graphs are even M\"obius ladders,
 and  the circular graph $C_{10}^{2}$; see Figure~\ref{k4}. 
Only odd wheels will be of relevance in this note.

A \emph{loose wheel} is a graph formed by a cycle $C$ 
and a new vertex $v\notin V(C)$ that has at least three
neighbours on $C$. A path of $C$ joining two neighbours of $v$ and which 
does not contain any other neighbour of $v$ is a \emph{segment} of the loose wheel.
The loose wheel is a
\emph{loose odd wheel} 
if $C$ has odd length, and if  at least three of the segments are odd
as well. Obviously, odd wheels are loose odd wheels.

To see that the presence of a loose odd wheel in the triangulation
actually certifies $t$-imperfection, we turn to a useful operation,
found by Gerards and Shepherd~\cite{GS98},
that preserves $t$-perfection. 
In a graph $G$, let $v$ be a vertex whose neighbourhood is stable 
(i.e., that does not induce any edge). Then, performing a \emph{$t$-contraction} at $v$
means contracting all the edges incident with $v$ (deleting parallel 
edges and loops which may arise). Results of \cite{GS98} show that {the resulting graph $G'$ will be $t$-perfect 
if $G$ is}. This implies:

\begin{lemma}\label{wheelem}
	Each graph which contains a loose odd wheel as an induced subgraph is $t$-imperfect.
\end{lemma}
\begin{proof}
	Since $t$-perfection is closed under taking induced subgraphs, we need only to check that a loose odd wheel is $t$-imperfect. We show that it can be $t$-contracted to an odd wheel, which is $t$-imperfect.
	
	The number $k$ of odd segments of a loose odd wheel is odd and at least~$3$. 
	We may use $t$-contraction to shrink each even segment to a single vertex and each odd segment to a single edge. It is easy to check that the graph obtained in this way is isomorphic to the odd wheel with $k+1$ vertices.
\end{proof}


\section{Proofs}\label{sec:proofs}

Let us first show that Theorem~\ref{hptrigthm} is a straightforward consequence of Theorem~\ref{trigthm}.
\begin{proof}[Proof of Theorem~\ref{hptrigthm}]
Perfect graphs are always $h$-perfect (see Section \ref{sec:defs}).

Conversely, 
let $T$ be an $h$-perfect plane triangulation.  
If $T$ is $K_4$-free, then $T$ is $t$-perfect and the conclusion follows from Theorem~\ref{trigthm}. 
Otherwise, $T$ is either equal to $K_4$, in which case it is perfect, or 
one of the triangles $X$ of any $K_4$ of $ T$ is separating. 
Thus, $X$ yields two proper induced subgraphs as in Lemma \ref{cliquecutset} and we are done by induction.
\end{proof}

We will use the following observation:

\begin{lemma}\label{perflem}
A plane triangulation without separating odd holes is perfect.
\end{lemma}
\begin{proof}
Such a triangulation $T$ does not contain any odd hole at all, as 
every face boundary is a triangle.
Moreover, $T$ does not contain any odd antihole
either: the only planar odd antihole is the antihole with five vertices, which 
is isomorphic to the cycle of length~$5$. 
Therefore, $T$ is perfect (this follows from 
the Strong Perfect Graph Theorem, in the easier special case of planar~\cite{Tucker73} or $K_4$-free graphs~\cite{Tucker77}). 
\end{proof}

The key ingredient in the proof of Theorem \ref{trigthm} is the following:
\begin{lemma}\label{sepoddholes}
Any plane triangulation that has an odd hole contains
either an induced loose odd wheel or a separating triangle.
\end{lemma}

We now use this to prove Theorem \ref{trigthm}:

\begin{proof}[Proof of Theorem \ref{trigthm}] 
\eqref{M1}$\Rightarrow$\eqref{M2} 
follows directly from Lemma~\ref{wheelem}.

To show \eqref{M2}$\Rightarrow$\eqref{M3}, 
suppose that $T$ is a minimal counterexample, that is, 
a plane triangulation without any induced loose odd wheel, which is not perfect and with 
$\card{V(T)}$ minimum. Since $K_4$ is an odd wheel, $T$ is $K_4$-free. 
By Lemma~\ref{perflem}, $T$ must have a separating odd hole,
and thus, by Lemma~\ref{sepoddholes}, also a separating triangle $X$. 
Clearly, $X$ yields two proper induced subgraphs as in Lemma \ref{cliquecutset}, which are perfect. But then $T$ is perfect, contrary to our assumption.


Finally, \eqref{M3}$\Rightarrow$\eqref{M1} follows from the polyhedral characterisation of perfect graphs and the obvious fact that perfect graphs cannot contain odd holes (see Section \ref{sec:defs}). 
\end{proof}

We now prove Lemma~\ref{sepoddholes}. 
The \emph{interior} of a cycle $C$ of a plane graph is the bounded component of $\mathbb R^2\sm C$.

\begin{proof}[Proof of Lemma \ref{sepoddholes}]
Let $T$ be a plane triangulation that contains an odd hole 
but no separating triangle. 
We need to show that $T$ contains an induced loose odd wheel.
Let $C$ be an odd hole of $T$ with a minimal number of vertices
in the interior of $C$.
 We claim the following:
\begin{equation}\label{degree}
\text{\emph{Some vertex in the interior of $C$ has at least three neighbours in $C$.}}
\end{equation}

We first show how to deduce that $T$ contains an induced loose odd wheel from~\eqref{degree}.
Let $v$ be a vertex in the interior of $C$ with at least three neighbours in $C$,
and denote by $W$  the induced loose wheel formed by $v$ and $C$. 
If $W$ is a loose odd wheel, we are obviously done. So, assume that $W$ has fewer than three
odd segments. 
Since $C$ is odd, this means that  $W$ has exactly one odd segment,~$P$ say. 

Let $v_1$ and $v_2$ be the two endvertices of $P$. 
For $i\in\set{1,2}$, there is exactly one segment $P_i$ which has $v_i$ as an endvertex and is not $P$. 
Let $v_i'$ be the other endvertex of $P_i$ (note that $v_1'$ and $v_2'$ are identical if $v$ has exactly
three neighbors on $C$).

Let $i\in\set{1,2}$ and let $D_i$ be the cycle formed by $vv_i$, $vv_{i}'$ and $P_i$. 
Let $F$ be the face boundary of the face of $T-v$ containing $v$. Note that $v$ is adjacent
to every vertex of $F$, as $T$ is a triangulation, and that $v_i$ and $v_i'$ lie in $F$. 
In particular, $F$ contains a path $Q_i$ from $v_i$ to $v_i'$ 
so that all its inner vertices lie in the interior of $D_i$.
Observe that  $Q_i$ is induced: a chord $xy$ of $Q_i$ 
would yield a separating triangle $xyv$, which we had excluded.
Since $v_i$ and $v_i'$ are not neighbours (otherwise $P_i$ would not have even length), 
it follows that the induced path $Q_i$ has length at least~$2$.
(Here, we also use that as a segment of the hole $C$, the path $P_i$ is induced.)

Now, if $Q_i$ is even then $Q_i$ and the odd path of $C$ joining $v_i$ to $v_i'$ 
form an odd hole of $T$ with fewer vertices in the interior than $C$,
contradicting the minimality of $C$. 

Hence, both $Q_1$ and $Q_2$ must be odd. Since $P$ is the unique odd segment of $W$, this implies that: $Q_1$, $Q_2$, the even path of $C$ joining $v_1'$ to $v_2'$ and $P$  form together an odd hole of $T$. We can see directly that the odd hole together with $v$ forms a loose odd wheel, and we are done.

\medskip

This shows that the lemma can be deduced from~\eqref{degree}. All that remains is to prove~\eqref{degree}.
For this purpose, suppose~\eqref{degree} to be false. That is, suppose 
that every vertex in the interior of $C$ has at most two neighbours on $C$.

Let $N=\card{V(C)}$ and let $v_1,\ldots,v_{N}$ be a circular ordering of the vertices of~$C$. 
In what follows, indices are always taken modulo $N$.

Since $T$ is a triangulation, 
the vertices $v_i$ and $v_{i+1}$ have a common neighbour $w_i$ in the interior of $C$ 
(for each $i\in\set{1,\ldots,N}$). By assumption, $w_i$ has no other neighbour in $C$ 
besides 
$v_i$ and $v_{i+1}$.

Observe that:
\begin{equation}\label{onlyw}
\begin{minipage}[c]{0.8\textwidth}\em
no vertex in the interior of $C$, except for $w_1,\ldots,w_N$,
has  two neighbours in $C$. 
\end{minipage}\ignorespacesafterend 
\end{equation} 
Indeed, 
if $u$ is another vertex in 
the interior of $C$ with at least (and then exactly) two neighbors $v_i$ and $v_j$ on $C$, then $v_i$ and $v_j$ cannot be consecutive on $C$ (because $T$ does not have a separating triangle). Therefore, the edges $uv_i$, $uv_j$ and the odd path of $C$ joining $v_i$ to $v_j$ form an odd hole with fewer vertices in the interior than~$C$,
which is impossible.

Next, let us see that:
\begin{equation}\label{defRi}
\begin{minipage}[c]{0.8\textwidth}\em
the neighbours of $v_{i+1}$ in the interior of $C$ form an  induced path $R_i$
from $w_{i}$ to $w_{i+1}$, for every $i\in\set{1,\ldots,N}$.
\end{minipage}\ignorespacesafterend 
\end{equation} 

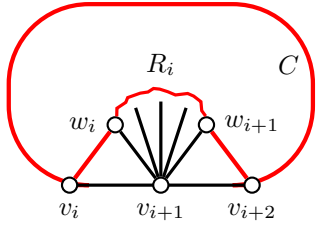
\begin{figure}[ht]
\centering
\begin{tikzpicture}[scale=0.8]

\def\xstep{1.5}
\def\offset{1}


\draw[rededge,rounded corners=30pt] (\offset+2*\xstep,0) -- (2*\offset+2*\xstep,0) -- (2*\offset+2*\xstep,3) -- (0,3) -- (0,0) -- (\offset,0);
\draw[hedge] (\offset,0) -- (\offset+2*\xstep,0);

\node at (2*\offset+2*\xstep-0.4,2) {$C$};

\node[hvertex,label=below:$v_i$] (vi) at (\offset,0){};
\node[hvertex,label=below:$v_{i+1}$] (vi+1) at (\offset+1*\xstep,0){};
\node[hvertex,label=below:$v_{i+2}$] (vi+2) at (\offset+2*\xstep,0){};

\node[hvertex,label=left:$w_i$] (wi) at (\offset+0.5*\xstep,1){};
\node[hvertex,label=right:$w_{i+1}$] (wi+1) at (\offset+1.5*\xstep,1){};

\draw[rededge] (vi) -- (wi);
\draw[hedge] (wi) -- (vi+1);
\draw[hedge] (vi+1) -- (wi+1);
\draw[rededge] (wi+1) -- (vi+2);

\draw[pathedge,red,bend left=80] (wi) to 
node[near start] (p1){} node[midway] (p2){} node[near end] (p3){}
(wi+1); 

\draw[hedge] (vi+1) edge (p1) edge (p2) edge (p3);

\node at (\offset+\xstep,2) {$R_i$};

\end{tikzpicture}
\caption{The path $R_i$ and the cycle $C'$ (in red)}\label{Rifig}
\end{figure}

Indeed, consider the face of $T-v_{i+1}$ containing $v_{i+1}$. Its face boundary consists
of neighbours of $v_{i+1}$ only (as $T$ is a triangulation), and contains a path $R_i'$ from 
$v_{i}$ to $v_{i+2}$ whose inner vertices lie in the interior of $C$. Moreover, $R_i'$ is induced since
any chord $xy$ would yield a separating triangle $xyv_{i+1}$. 

Since $w_{i}v_{i}$ and $w_{i+1}v_{i+2}$ are edges of $T$,
 $w_{i}$ succeeds
$v_{i}$ and $w_{i+1}$ precedes $v_{i+2}$ in $R_i'$.
Choose $R_i$ as the subpath of $R_i'$ from $w_{i}$ to $w_{i+1}$ (see Figure~\ref{Rifig}). 
This proves~\eqref{defRi}.

Let us note rightaway that for each $i\in\set{1,\ldots,N}$:
\begin{equation}\label{Riind}
\emtext{
inner vertices of $R_i$ have no other neighbour in $C$ than $v_{i+1}$.
}
\end{equation}
Suppose that $r$ is an inner vertex of $R_i$ that is adjacent to $v_j\neq v_{i+1}$. 
Then $r$ has two neighbours in $C$ and thus, by~\eqref{onlyw}, has to be 
one of $w_1,\ldots, w_N$. Since for every $l\in\set{1,\ldots, N}$, $w_\ell$ has exactly $v_\ell$ and $v_{\ell+1}$
as neighbours in $C$, it follows that $r=w_i$ or $r=w_{i+1}$. This is impossible 
since  $r$ is an inner vertex of $R_i$.

Next:
\begin{equation}\label{oddlength}
\emtext{
each  $R_i$ has odd length.
}
\end{equation}
If $R_i$ was even then $C-v_{i+1}$, $R_i$ and the edges $v_iw_i$ and $w_{i+1}v_{i+2}$ would form together 
an odd hole $C'$ (\eqref{Riind} shows that $C'$ is induced).
Clearly, $C'$ has 
  fewer vertices in its interior than $C$, and this contradicts the minimality of $C$; see Figure~\ref{Rifig}.

We need one more fact. For each $i,j\in\set{1,\ldots, N}$:
\begin{equation}\label{wiwj}
\emtext{
if $w_iw_j\in E(T)$ then $i=j+1$ or $j=i+1$.
}
\end{equation}
Suppose that $i\notin\{j-1,j+1\}$. Then, $C$ contains two disjoint paths, 
one between $v_{i+1}$ and $v_j$, and the other between $v_{j+1}$ and $v_i$.
One of these paths has even length, and in particular length at least~$2$;
denote the path by $P$. We extend $P$ by  the two edges between its endvertices
and $w_i$ and $w_j$, and finally add the edge $w_iw_j$. The resulting cycle
is induced and of odd length at least~$5$, which is impossible as it 
contradicts  the minimal choice of $C$.

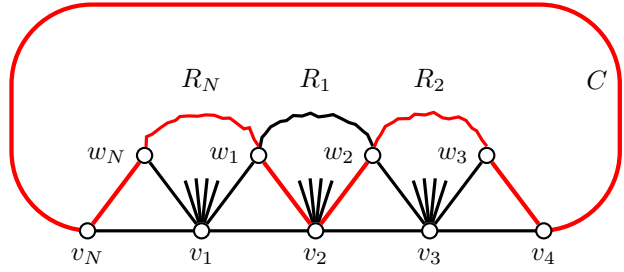
\begin{figure}[ht]
\centering
\begin{tikzpicture}[scale=1]

\def\xstep{1.5}
\def\offset{1}


\draw[rededge,rounded corners=30pt] (\offset+4*\xstep,0) -- (2*\offset+4*\xstep,0) -- (2*\offset+4*\xstep,3) -- (0,3) -- (0,0) -- (\offset,0);
\draw[hedge] (\offset,0) -- (\offset+4*\xstep,0);

\node at (2*\offset+4*\xstep-0.3,2) {$C$};

\node[hvertex,label=below:$v_N$] (vN) at (\offset,0){};
\node[hvertex,label=below:$v_1$] (v1) at (\offset+1*\xstep,0){};
\node[hvertex,label=below:$v_2$] (v2) at (\offset+2*\xstep,0){};
\node[hvertex,label=below:$v_3$] (v3) at (\offset+3*\xstep,0){};
\node[hvertex,label=below:$v_4$] (v4) at (\offset+4*\xstep,0){};

\node[hvertex,label=left:$w_N$] (wN) at (\offset+0.5*\xstep,1){};
\node[hvertex,label=left:$w_1$] (w1) at (\offset+1.5*\xstep,1){};
\node[hvertex,label=left:$w_2$] (w2) at (\offset+2.5*\xstep,1){};
\node[hvertex,label=left:$w_3$] (w3) at (\offset+3.5*\xstep,1){};

\draw[hedge] (v1) edge +(72:0.7) edge +(84:0.7) edge +(96:0.7) edge +(108:0.7);
\draw[hedge] (v2) edge +(72:0.7) edge +(84:0.7) edge +(96:0.7) edge +(108:0.7);
\draw[hedge] (v3) edge +(72:0.7) edge +(84:0.7) edge +(96:0.7) edge +(108:0.7);

\draw[rededge] (vN) -- (wN);
\draw[hedge] (v1) -- (wN);
\draw[hedge] (v1) -- (w1);
\draw[rededge] (v2) -- (w1);
\draw[rededge] (v2) -- (w2);
\draw[hedge] (v3) -- (w2);
\draw[hedge] (v3) -- (w3);
\draw[rededge] (v4) -- (w3);

\draw[pathedge,red,bend left=80] (wN) to (w1); 
\draw[pathedge,bend left=80] (w1) to (w2); 
\draw[pathedge,red,bend left=80] (w2) to (w3); 

\node at (\offset+\xstep,2) {$R_N$};
\node at (\offset+2*\xstep,2) {$R_1$};
\node at (\offset+3*\xstep,2) {$R_2$};

\end{tikzpicture}
\caption{The cycle $C_1$ (in red)}\label{finalargumentfig}
\end{figure}

\medskip
Now, consider the walk
$R:=w_1R_1w_2R_2w_3\ldots w_NR_Nw_1$. By~\eqref{defRi} and~\eqref{Riind}, 
$R$ is a cycle of length
at least~$5$. By~\eqref{oddlength}, it is odd. 
Since there are fewer vertices in the interior of $R$ than in the interior of $C$, 
the minimality of $C$ ensures that $R$ has a chord.

As every $R_i$ is induced (by~\eqref{defRi})
and as vertices $w_i$ and $w_j$ may be adjacent only if $|i-j|= 1$ (as stated by~\eqref{wiwj}),
such a chord must have at least one endvertex that is an inner vertex
of some $R_i$, say $R_1$. In particular, $R_1$ is not an edge.

Let $C_1$ be  the  cycle formed by $v_2w_2$, $R_2$,  $w_3 v_4$, $v_Nw_N$, $R_N$, $w_1v_2$ and the 
odd $v_4$--$v_N$~path of $C$ (see Figure~\ref{finalargumentfig}). 
By~\eqref{oddlength}, $C_1$ has odd length. 
Obviously, $C_1$ is not a triangle and has fewer vertices in its interior than $C$. Hence, it
must have a chord $xy$. 

By~\eqref{Riind} and as the $R_i$ are induced, 
this is only possible if one of $x$ and $y$, say $x$, lies in $R_2$ while $y$ 
lies in $R_N$. Furthermore, $xy$ lies in the interior of $R$ (except its ends).

It is easy to check that $\set{x,y}\cap\set{w_1,w_2}\neq\es$: otherwise, $v_1xyv_3v_2v_1$ is an odd hole (of length~$5$), contradicting the minimality of $C$.

Therefore, and by symmetry, we may assume without loss of generality that $x=w_2$. 
Since $R_1$ is induced and not an edge, 
this implies $y\neq w_1$. 
For later use, we note the following trivial consequence  (see Figure~\ref{anotherfig}):
\begin{equation}\label{cycleD}
\begin{minipage}[c]{0.8\textwidth}\em
the interior of the cycle $D=v_1v_2w_2yv_1$ contains $R_1-w_2$ but no vertex of $C$.  
\end{minipage}\ignorespacesafterend 
\end{equation}

\begin{figure}[ht]
\centering
\begin{tikzpicture}[scale=1]

\def\xstep{1.5}
\def\offset{1}


\node[hvertex,label=left:$v_{N-1}$] (vN-1) at (0,1){};

\node[hvertex,label=below:$v_N$] (vN) at (\offset,0){};
\node[hvertex,label=below:$v_1$] (v1) at (\offset+1*\xstep,0){};
\node[hvertex,label=below:$v_2$] (v2) at (\offset+2*\xstep,0){};
\node[hvertex,label=below:$v_3$] (v3) at (\offset+3*\xstep,0){};
\node[hvertex,label=below:$v_4$] (v4) at (\offset+4*\xstep,0){};

\begin{scope}[on background layer]
\draw[rededge,rounded corners=30pt] (\offset+4*\xstep,0) -- (2*\offset+4*\xstep,0) -- (2*\offset+4*\xstep,3) -- (0,3) -- (vN-1);
\draw[hedge,rounded corners=30pt] (vN-1) -- (0,0) -- (v3);
\end{scope}

\draw[rededge] (v3) -- (v4);

\node at (2*\offset+4*\xstep-0.3,2) {$C$};

\node[hvertex,label=left:$w_N$] (wN) at (\offset+0.5*\xstep,1){};
\node[hvertex,label=left:$w_1$] (w1) at (\offset+1.5*\xstep,1){};
\node[hvertex,label=left:$w_2$] (w2) at (\offset+2.5*\xstep,1){};
\node[hvertex,label=left:$w_3$] (w3) at (\offset+3.5*\xstep,1){};

\draw[hedge] (v1) edge +(70:0.7) edge +(80:0.7)  edge +(108:0.7);
\draw[hedge] (v2) edge +(72:0.7) edge +(84:0.7) edge +(96:0.7) edge +(108:0.7);
\draw[hedge] (v3) edge +(72:0.7) edge +(84:0.7) edge +(96:0.7) edge +(108:0.7);

\draw[hedge] (vN) -- (wN);
\draw[rededge] (v1) -- (wN);
\draw[rededge] (v1) -- (w1);
\draw[hedge] (v2) -- (w1);
\draw[hedge] (v2) -- (w2);
\draw[rededge] (v3) -- (w2);
\draw[hedge] (v3) -- (w3);
\draw[hedge] (v4) -- (w3);

\node[hvertex,label=below left:$y$] (y) at (\offset+1*\xstep,1.5){};

\begin{scope}[on background layer] 
\draw[hedge, rounded corners=5pt] (w2) -- ++(0,0.7) -- ++(-1.5*\xstep,0) -- (y);
\end{scope}

\draw[hedge] (v1) -- (y);

\draw[pathedge,bend left=40] (wN) to (y);
\draw[pathedge,bend left=40] (y) to (w1); 

\draw[pathedge,red,bend left=80] (w1) to (w2); 
\draw[pathedge,bend left=80] (w2) to (w3); 

\node at (\offset+\xstep,2) {$R_N$};
\node at (\offset+2*\xstep,2) {$R_1$};
\node at (\offset+3*\xstep,2) {$R_2$};

\node[hvertex] (wN-1) at (0.7,1.5){};

\draw[rededge] (vN-1) -- (wN-1);
\draw[hedge] (vN) -- (wN-1);
\draw[pathedge,red,bend left=30] (wN-1) to (wN); 
\node at (\offset+0.5,1.7) {$R_{N-1}$};

\end{tikzpicture}
\caption{The cycle $C_N$ (in red)}\label{anotherfig}
\end{figure}
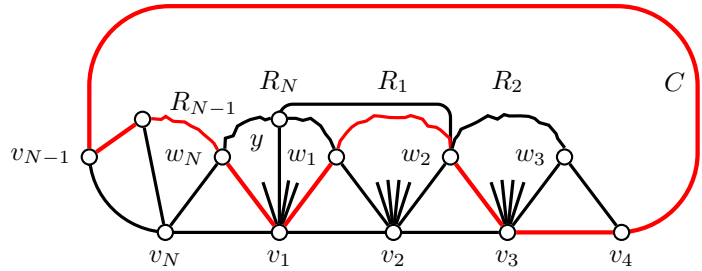

Observe 
moreover that $y\neq w_N$ by~\eqref{wiwj}.
In particular, this means that $y$ is an inner vertex of $R_N$, which thus 
cannot be an edge either.
Let $C_N$ be the odd cycle formed 
by  $v_1w_1$, $R_1$, $w_{2}v_{3}$, $v_{N-1}w_{N-1}$, $R_{N-1}$, $w_Nv_1$, and the 
odd $v_3$--$v_{N-1}$~path of $C$. That is, we construct $C_N$ in the exact way as
$C_1$, only for $R_N$ instead of $R_1$. 

As for $C_1$, the cycle $C_N$ has a chord
that joins a vertex  $s$ of  $R_{N-1}$ 
to a vertex $t$ of $R_1$. Moreover, we deduce in the same way as for $C_1$ that
 $\set{s,t}\cap\set{w_1,w_N}\neq\es$. 

Since we can reach $C$ from $s$ without
meeting the cycle $D$ (using the path $R_{N-1}w_{N-1}v_{N-1}$), it follows that $s$ does not lie in the interior of~$D$. On the other hand $R_1-w_2$ lies in the interior of $D$ (see~\eqref{cycleD}), thus $t=w_2$. As $\set{s,t}\cap\set{w_1,w_N}\neq\es$, we must have $s=w_N$, but this  contradicts~\eqref{wiwj}. 
We have reached the final contradiction that proves~\eqref{degree} and therefore
the lemma.

\end{proof}

\section{Conclusion}

Theorem~\ref{hptrigthm}, that perfection and $h$-perfection are the same in 
plane triangulations, puts an end to further investigations of $h$-perfect 
plane triangulations: most of the interesting questions can be answered
by appealing to their perfection.

In the larger class of planar graphs, however, perfection and $h$-perfection are no longer
equivalent. Evidently, a $5$-cycle is $h$-perfect but not perfect. 

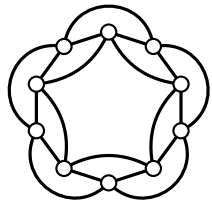
\begin{figure}[ht]
\centering
\begin{tikzpicture}
\begin{scope}
\def\angle{360/10}
\def\radius{1cm}
\foreach \i in {0,2,4,6,8}{
  \begin{scope}[on background layer] 
    \draw[hedge] (90+\angle*\i:\radius) -- (90+\angle+\angle*\i:\radius);
    \draw[hedge,bend left=30] (90+\angle*\i:\radius) to (90+2*\angle+\angle*\i:\radius);
  \end{scope}
  \node[hvertex] (v\i) at (90+\angle*\i:\radius){};
}
\foreach \i in {1,3,5,7,9}{
  \begin{scope}[on background layer] 
    \draw[hedge] (90+\angle*\i:\radius) -- (90+\angle+\angle*\i:\radius);
    \draw[hedge,radius=1] (90+\angle*\i:\radius) 
       .. controls (90+0.5*\angle+\angle*\i:1.6*\radius) and (90+1.5*\angle+\angle*\i:1.6*\radius) ..
          (90+2*\angle+\angle*\i:\radius);
  \end{scope}
  \node[hvertex] (v\i) at (90+\angle*\i:\radius){};
}
\end{scope}
\end{tikzpicture}
\caption{The graph $C_{10}^2$.}\label{k4}
\end{figure}

Planarity is closed under taking $t$-contractions and induced subgraphs. 
Therefore, it might  be possible to characterise those 
$t$-imperfect planar graphs that are minimal with respect to these two operations.
So far, the known planar minimal obstructions are the odd wheels and $C_{10}^{2}$ (see Figure \ref{k4}). 
Are these the only ones?\sloppy


While general planar graphs obviously offer  less structure than triangulations, 
some arguments might still be adapted. In particular, 
the key argument to derive Theorem~\ref{hptrigthm} from Theorem~\ref{trigthm} 
 remains valid, and thus yields:
\begin{proposition}
	A planar graph is $h$-perfect if and only if each of its $K_4$-free induced subgraphs is $t$-perfect.
\end{proposition}

Therefore, as for plane triangulations, any characterisation of $t$-perfection in planar graphs in terms of forbidden induced subgraphs can be directly extended to $h$-perfection.

\bibliographystyle{amsplain}
\bibliography{bibliography}

\vfill
\small
\vskip2mm plus 1fill
\noindent
Version \today
\bigbreak

\noindent
Yohann Benchetrit
{\tt <yohann.benchetrit@ulb.ac.be>}\\
Universit\'e libre de Bruxelles, Belgium\\[3pt]
Henning Bruhn
{\tt <henning.bruhn@uni-ulm.de>}\\
Universit\"at Ulm, Germany\\[3pt]

\end{document}